\documentclass{article}
\usepackage{lmodern,bbm}
\usepackage{amsmath}
\usepackage{amsthm}
\usepackage{amssymb}
\usepackage{mathrsfs}
\usepackage{esint}
\usepackage{enumitem}
\usepackage[colorlinks]{hyperref}
\hypersetup{
    colorlinks=true,
    linkcolor=blue,
    citecolor=blue,      
}
\usepackage{cleveref}
\usepackage{amsfonts}
\usepackage{comment}
\usepackage{authblk}

\DeclareMathOperator\R{\mathbb{R}}

\DeclareMathOperator\M{\mathcal{M}}
\DeclareMathOperator\argmin{argmin}
\newcommand*\prob{\mathop{}\!\mathscr{P}}
\newcommand{\p}{\varrho}
\newcommand*\diff{\mathop{}\!\mathrm{d}}

\newtheorem{theorem}{Theorem}[section]
\newtheorem{lemma}[theorem]{Lemma}

\theoremstyle{definition}
\newtheorem{definition}[theorem]{Definition}

\theoremstyle{remark}
\newtheorem{remark}[theorem]{Remark}

\numberwithin{equation}{section}

\title{Fisher information and continuity estimates for nonlinear but 1-homogeneous diffusive PDEs (via the JKO scheme)}
\author[1]{Thibault Caillet}
\author[1]{Filippo Santambrogio}

\affil[1]{\scriptsize Universite Claude Bernard Lyon 1, ICJ UMR5208, CNRS, Ecole Centrale de Lyon, INSA Lyon, Universit\'e Jean Monnet,
69622 Villeurbanne, France. {\tt caillet,santambrogio@math.univ-lyon1.fr}}
\date{\today}

\begin{document}

\maketitle
\begin{abstract} 
In this short paper we prove, using the JKO scheme, that quantities such as the Fisher information stay bounded or decrease across time for a family of $1$-homogeneous diffusive PDEs. As a corollary, we prove that moduli of continuity are conserved across time for the solutions of those PDEs. 
\end{abstract}

\section{Introduction}
The now classical JKO scheme, introduced in the seminal paper \cite{JKO}, and generalized later by many authors, is an approximation scheme used to construct solutions to nonlinear diffusion equations of the type 
\begin{align*}
    \partial_t \p = \nabla \cdot (\p \nabla c^*(\nabla f'(\p))),
\end{align*}
where $c:\R^d \to \R_+ \cup \{+\infty\}$ and $f : \R_+ \to \R\cup \{+\infty\}$  are convex functions. Among examples of PDEs which have this form, we can obtain equations of the form $\partial_t\p=\Delta_p(\phi(\rho))$ (which are sometimes written in the form $\partial_t(\beta(u))=\Delta_p u$, using $u=\phi(\rho)$ and $\beta=\phi^{-1}$). These equations are usually called {\it doubly nonlinear}: they are a combination of porous-medium type equations (of the form $\partial_t\rho=\Delta(\phi(\rho))$, see \cite{Vaz}) and parabolic $p-$Laplacian equations (of the form $\partial_t u=\Delta_p u$).

The (generalized) JKO scheme consists in building a time discrete approximation of such equations, which is of the form 
\begin{align*}
        \p_{k+1} ^h =\argmin_{\rho} W_{c_h}  (\rho,\p_k ^h) + \mathscr{F}(\rho),
\end{align*}
where $h>0$ is the discrete time step, $W_{c_h}$ is the transport cost for the cost $c_h = hc(\frac{\cdot}{h})$ (when $c(z)=\frac 1p|z|^p$ we obtain $W_{c_h}=\frac{1}{ph^{p-1}}W_p^p$) and $\mathscr{F}(\p) = \int f(\p(x)) \diff x$. 
To prove the convergence of this scheme, it is often useful to derive estimates on the discrete solutions $(\p^h _k)$ to gain some compactness that allows to pass to the limit, as it is done for example in \cite{Agueh,Otto,ContractionPrinciple} and many other papers. Some other estimates, derived on the discrete scheme, may be used to study the behaviour of the limit continuous PDE, see for instance \cite{IacPatSan,FokkerPlanckLp,CailletSantambrogioDoubly}. In \cite{L2H2Convergence}, the authors use discrete estimates obtained on the scheme to prove it stronger convergence in the case of the Fokker-Planck equation than what was originally proven. This proof used, for instance, a bound on the H\"{o}lder norms, obtained from a non-optimal $W^{1,p}$ assumption on the initial condition.

We restrict our analysis to the case where $\mathscr{F}(\p) = \mathscr{E}(\p) = \int \p \log (\p) \diff x$, meaning $ \nabla f'(\p)=\nabla \log(\p)$. This equation is in some sense the most linear one, for fixed cost $c$, among the doubly non-linear equations that we presented at the beginning. Its almost linear behavior lies in the fact that with this logarithmic choice the equation becomes 1-homogeneous: if $\rho$ is a solution, then $\lambda\rho$ as well for any $\lambda>0$. This class of equations has already been widely studied, and we cite for instance \cite{Ish} for an existing result( the relevant case in such a paper is case II), \cite{DPDolGen} for its connections with functional inequalities, and \cite{Agu03,Agu08} for the study of the trend to equilibrium (the 1-homogeneous case is called in \cite{Agu08} ``Generalized heat equation''. As an example of this class of equations, we can think at the following one: $\partial_t\rho=\Delta_p(\rho^{q-1}),$ where $p$ and $q$ are conjugate exponents (so that $(p-1)(q-1)=1$).

For this very class of PDEs we can adapt results from \cite{FokkerPlanckLp} which readily generalize in our context thanks to the recent extension of the so called "five gradients inequality" (see \cite{FiveGradsIneq,DiMMurRad}) to other transport costs than the quadratic one. We prove that the Fisher information, as well as many generalizations of this quantity, decreases along the scheme, and therefore also decreases in time for the limit PDE. We underline that a direct computation of the time-derivative of the Fisher information on the limit equation does not show in a straightforward way that these quantities decrease in time. This reminds a similar result (on very different equations) recently presented in \cite{guillen2023landau}: in such a paper the authors shows that the fisher information decreases but their analysis is based on clever tricks as its time derivative was not easily seen to be negative.

As a limit case, we also obtain that the Lipschitz constant of $\log\rho$ is preserved across the iterations. This in turn implies that any modulus of continuity of $\log\rho$ is also preserved along iterations of the scheme, and therefore pass to the limit in the time continuous PDE.  This analysis is done first for more classical cost functions $c$ which are finite everywhere and then extended to the so-called "relativistic" costs used for example in the construction of solutions to the relativistic heat equation (see \cite{McCPue})
\begin{align*}
    \partial_t \p = \nabla \cdot \left( \p \,\frac{\nabla \p}{\sqrt{\p ^2 + |\nabla \p|^2}} \right) .
\end{align*}

In the whole paper, we will assume that $\Omega$ is the closure of an open bounded convex subset of $\R^d$.
We take $c: \R^d \to \R_+$ to be $C^1$, strictly convex, radially symmetric and such that $c(0)=0$. Radial symmetry can be removed from the assumptions, as it will be explained later, and the finiteness of the cost will be removed in Section 3.2.

\section{Preliminaries}

We refer to \cite{AGS,OTAM,VillaniO&N} for the general theory of optimal transport and its application to gradient flows, and we compile below a selection of helpful facts we will use in the sequel. 

\begin{theorem}
    \label{everythingtheorem}
    Let $\p, g \in \M _+(\Omega)$ be two non-negative measures on $\Omega$. The following statements are classical~:
    \begin{enumerate}
        \item The problem
        \begin{align*}
            W_c(\mu,\nu) := \min \left\{\int_{\Omega\times \Omega} c(x-y) \diff \gamma \quad ; \quad \gamma \in \Pi(\mu,\nu) \right\},
        \end{align*}
        where $\Pi(\mu,\nu)$ is the set of non-negative measures on $\Omega\times \Omega$ with first marginal $\mu$ and second marginal $\nu$, admits a solution $\gamma^*$. If $\mu=\p \diff x$, with $\p \in L^1 (\Omega)$, then the solution is unique, and given by $\gamma^*=(\textnormal{id},T)_\# \p$ for some $T: \Omega \to \Omega$ which is called the optimal transport map. 
        \item We have 
        \begin{align*}
             W_c (\mu, \nu) = \max \left\{ \int_\Omega \varphi \diff \mu + \int_\Omega \psi \diff \nu \quad ; \quad \varphi(x) + \psi(y) \leq c(x-y)  \; \forall x,y\in \Omega \right\}.
        \end{align*}
        The optimal $\varphi$ and $\psi$ are called Kantorovich potentials, and they can be taken $c$-concave, satisfying
        \begin{align*}
            &\varphi(x)=\inf_{y\in \Omega} c(x-y) - \psi(y),
            &\psi(y)=\varphi^c (y):=\inf_{x\in \Omega} c(x-y) - \varphi(x),
        \end{align*}
        from which one can prove that $\varphi$ and $\varphi^c$ are Lipschitz continuous. 
        \item Since $c$ is strictly convex, $c^* \in C^1$ and the optimal transport map $T$ is given by
        \begin{align*}
            T(x) = x- \nabla c^*(\nabla \varphi (x)),
        \end{align*}
        where $\varphi$ is any Kantorovich potential. It is well defined almost everywhere since a Lipschitz function is differentiable almost everywhere.
    \end{enumerate}
\end{theorem}

\begin{definition}
    We define $\mathscr{E}: \prob{} (\Omega) \to \R \cup \{+\infty\}$ by 
    \begin{align*}
        \mathscr{E}(\mu) := 
        \begin{cases}
            \int_\Omega \p \log (\p) \diff x \quad &\text{if } \p \ll \diff x \text{ and } \p =\frac{\diff \mu}{\diff x} \\
            +\infty &\text{otherwise}
        \end{cases}
    \end{align*}
\end{definition}
Since we will deal only with absolutely continuous measures, we will use the usual abuse of notation of denoting $\mu$ by its density $\p$.

\begin{theorem}
\label{propsPi}
Let $g\in \M(\Omega)$ be a non-negative measure, and set $c_h(z):= hc(z/h)$. The problem
    \begin{align}
\label{jkoproblem}
\p \in \argmin_{\rho} W_{c_h}  (\rho,g) + \mathscr{E}(\rho)
\end{align}
admits a unique solution which we denote $\Pi(g)$. Given $g_1,g_2$ two non-negative measures, if $g_1 \leq g_2$, then $ \Pi(g_1) \leq \Pi(g_2)$, and for $\lambda>0$, $\Pi(\lambda g)=\lambda \Pi(g)$.
\end{theorem}
\begin{proof}
    Existence and uniqueness follow from the direct method in the calculus of variations, since both functionals in the minimization are lower semi-continuous and convex, and $\mathscr{E}$ is strictly convex, see for example \cite{Agueh}. The second claim is a direct application of the $L^1$ contraction principle proved in Theorem 1.3 from \cite{ContractionPrinciple}. For the last claim, we have
    \begin{align*}
        \inf_{\p \in \M_+ (\Omega), \int \p = \int \lambda g} W_{c_h}  (\p,\lambda g) + \mathscr{E}(\p) &= \inf_{\p \in \M_+ (\Omega), \int \p = \int g} W_{c_h}  (\lambda \p,\lambda g) + \mathscr{E}(\lambda \p) \\
        &= \lambda \left(\inf_{\p \in \M_+ (\Omega), \int \p = \int g} W_{c_h}  ( \p,g) + \mathscr{E}( \p)  \right) + \lambda \log(\lambda) \int_\Omega g \diff x,
    \end{align*}
    so that the solution is $\lambda \Pi(g)$.
\end{proof}

The following lemma, proved in \cite{FiveGradsIneq}  (see also \cite{DiMMurRad}) as a generalization of the results of \cite{BVEstimates}, is useful to prove that suitable quantitites decrease along the steps of the JKO scheme.

\begin{lemma}
\label{5grads}
Let $\Omega \subset \R^d$ be bounded and convex with non-empty interior, $\varrho,g \in W^{1,1}(\Omega)$ be two probability densities and $H \in C^1(\R^d \backslash \{0\})$ be a radially symmetric convex function, then the following inequality holds 
\begin{equation}
    \int_{\Omega} \big( \nabla \varrho \cdot \nabla H ( \nabla \varphi) + \nabla g \cdot \nabla H (\nabla \psi) \big) \diff x \geq 0,
\end{equation}
where $(\varphi,\psi)$ is a choice of Kantorovich potentials for the optimal transport problem between $\varrho$ and $g$ for the transport cost given by $c$, with the convention that $\nabla H (0) = 0$.
\end{lemma}

Using \Cref{5grads}, one can prove the following :

\begin{lemma}
\label{decreasingH}
    Let $H$ be a radially symmetric convex function, $g\in L^1(\Omega)$ be non negative and denote by $\p$ the solution of \cref{jkoproblem}. Then we have 
    \begin{align*}
        \int_\Omega H(\nabla \log g) \diff g \geq \int H (\nabla \log \p) \diff \p 
    \end{align*}
\end{lemma}
\begin{proof}
We follow the proof of Proposition 5.1 from \cite{FokkerPlanckLp}. We start from the case where $H\in C^1(\R^d \backslash \{0\})$. We have from the fact that $H$ is convex,
\begin{align*}
    \int_\Omega H (\nabla \log g) \diff g \geq \int_\Omega H( \nabla \psi ) \diff g + \int_\Omega \nabla H ( \nabla \psi ) \cdot \left( \nabla \log g - \nabla\psi)\right) \diff g
\end{align*}
Using $g= T_\# \p$ and $\nabla \psi \circ T =-\nabla \varphi$, together with the optimality condition of \cref{jkoproblem} $\nabla \varphi =-\nabla \log \p$, we have 
\begin{align*}
    \int_\Omega H( \nabla \psi ) \diff g =\int_\Omega H ( \nabla \psi \circ T) \diff \p = \int_\Omega H(-\nabla \varphi) \diff \p = \int_\Omega H( \nabla \log \p) \diff \p ,
\end{align*}
and
\begin{align*}
    \int_\Omega \nabla H ( \nabla \psi ) \cdot \nabla \psi \diff g = - \int_\Omega \nabla  H (- \nabla \varphi ) \cdot \nabla \varphi \diff \p = \int_\Omega H (-\nabla \varphi) \cdot \nabla \p \diff x.
\end{align*}
Using \Cref{5grads}, we have
\begin{align*}
    \int_\Omega \nabla H ( \nabla \psi ) \cdot \left( \nabla \log g \right) \diff g + \int_\Omega\nabla  H (\nabla \varphi) \cdot \nabla \p \diff x =  \int_\Omega \nabla H ( \nabla \psi ) \cdot \nabla g \diff x  + \int_\Omega \nabla H (\nabla \varphi) \cdot \nabla \p \diff x \geq 0, 
\end{align*}
so that we finally conclude 
\begin{align*}
    \int_\Omega H(\nabla \log g) \diff g \geq \int H (\nabla \log \p) \diff \p.
\end{align*}
By approximation we can get rid of the assumption $H\in C^1(\R^d \backslash \{0\})$ by passing the above inequality to the limit and obtain the result for any convex and radially symmetric function $H$.
\end{proof}

\begin{remark}
    We observe that the assumption that $H$ is radially symmetric can be weakened. $H$ was supposed to be radial in \cite{FiveGradsIneq} in order to handle the boundary term in the integration by parts which is needed in the proof of Lemma
\ref{5grads}, but it is observed that we do not need neither $c$ nor $H$ to be radial, but only $\nabla c^*$ to be parallel to $\nabla H$. Then we observe that in our proof (and in that of \cite{FiveGradsIneq}), we freely moved the minus sign inside and outside of $\nabla H$, assuming at least $H$ to be even. This assumption can also be removed, but in this case the five-gradients inequality becomes 
$$
    \int_{\Omega} \big( \nabla \varrho \cdot \nabla H ( \nabla \varphi) - \nabla g \cdot \nabla H (-\nabla \psi) \big) \diff x \geq 0.
$$
With this modification, the result of Lemma \ref{decreasingH} still holds.
\end{remark}

\section{Estimates on steps of the JKO Scheme}
\subsection{Finite costs}
\begin{definition}
\label{defScheme}
For $h>0$ and $\p_0 \in L^1(\Omega)$ such that $\mathscr{E}(\p_0)  < + \infty$, we recursively define the sequence $(\p_k ^h)_k$ by
\begin{align*}
        \p_{k+1} ^h =\argmin_{\rho} W_{c_h}  (\rho,\p_k ^h) + \mathscr{E}(\rho),
\end{align*}
which is well defined because of \Cref{propsPi}
\end{definition}

The following theorem is taken from \cite{FokkerPlanckLp} and can directly be generalized to our context thanks to \Cref{5grads,decreasingH} with the same proof.

\begin{theorem}
\label{UnifEstimates}
    Let $\p_0\in L^1(\Omega)$.  We have
    \begin{align*}
        \int H (\nabla \log \p_0) \diff \p_0 \geq \int H (\nabla \log \p_k ^h) \diff \p_k ^h  
    \end{align*}
    for all $h,k$, where $H$ is a radially symmetric and convex function. In particular, the following estimates are uniform in $h$ and $k$:
    \begin{enumerate}
        \item if $\log(\p_0)$ is $L$-Lipschitz, then $\log(\p_k ^h)$ is $L$-Lipschitz;
        \item for $p>1$, if $\p_0^{1/p} \in W^{1,p} (\Omega)$, then $(\p_k ^h)^{1/p} \in W^{1,p} (\Omega)$ and $\Vert (\p_k ^h)^{1/p} \Vert_{W^{1,p}} \leq \Vert \p_0^{1/p} \Vert_{W^{1,p}}$;
        \item if  $\p_0 \in BV(\Omega)$, then $\p_k ^h \in BV(\Omega)$ and $\Vert \p_k ^h \Vert_{BV} \leq \Vert \p_0 \Vert_{BV}$.
        \item if $\p_0 \in W^{1,1} (\Omega)$, all $\p_k ^h$ belong to a weakly compact subset of $W^{1,1} (\Omega)$.
    \end{enumerate}
\end{theorem}
\begin{proof}
    We only give the main ideas, as the details can be found in \cite{FokkerPlanckLp}. The first statement is a direct consequence of \Cref{decreasingH}, iterated along the steps of the scheme. For 1., we pick $H$ to be the indicator of the centered ball of radius $L$ (by indicator we mean here the function which is $0$ on the ball and $+\infty$ outside it). For 2. we take $H(z)=|z|^p$ with $p>1$, and 3. is covered by the case $p=1$. For 4., one can use the Dunford-Pettis theorem so find a superlinear convex function $H$ such that $H(\nabla \log (\p_0)) \in L^1(\p_0)$, which decreases at each step, thus providing the weak compactness.
\end{proof}

\begin{remark}
    Many authors use bounds on the generalized Fisher information to help prove the convergence of the JKO scheme to the limit PDE, see for example \cite{Agueh,McCPue}. However, most often this bounds are integral in time, i.e. only bounds on $\int_0^T\int H(\nabla\p_t)\diff\p_t\diff t$ are used, and hence proven. 
\end{remark}

The above theorem includes that fact that the Lipschitz constant of the logarithm is preserved along the iterations of this JKO scheme. We want now to extend this to other moduli of continuity. We will use the following abstract fact. 
\begin{theorem}
\label{proofOperator}
    Let $\pi : L^\infty (\Omega) \to  L^\infty (\Omega)$ be an operator such that 
    \begin{enumerate}
        \item For all $u,v\in L^\infty (\Omega)$, if $u\geq v$, then $\pi(u)\geq \pi(v)$.
        \item For all $\lambda\in\mathbb{R}$ and $u\in L^\infty(\Omega)$, $\pi(\lambda+ u)=\lambda+ \pi(u)$.
        \item For all $u\in L^\infty(\Omega)$, if $u$ is $k$-Lipschitz, then $\pi(u)$ is also $k$-Lipschitz.
    \end{enumerate}
    Then, if $u$ admits a concave modulus of continuity $\omega$, then  $\pi(u)$ admits the same modulus of continuity.
\end{theorem}

\begin{proof}
    If $u$ admits $\omega$ as a modulus of continuity, we start by approximating it with $L$-Lipschitz functions : for $x \in \Omega$, set
    \begin{align*}
        u_L (x) = \inf_{y \in \Omega}\; L|x-y| + u(x),
    \end{align*}
    so that $u_L$ is $L$-Lipschitz. Of course $u_L$ satisfies $u_L\leq u$ (by taking $y=x$). Furthermore, from the inequality
    \begin{align*}
        u (x) - u (y) \leq \omega(|x-y|),
    \end{align*}
    we deduce that we have 
    \begin{align*}
        L|x-y| + u (y) \geq L |x-y| + u (x) - \omega (|x-y|),
    \end{align*}
    so that, passing to the $\inf$ in $y$, we obtain
    \begin{align*}
        u_L (x) \geq u (x) + \alpha(L),
    \end{align*}
     where $\alpha(L)=\inf_{r>0}  Lr-\omega(r)\leq 0$. We can therefore conclude that
    \begin{align*}
        u \geq u_L \geq u+\alpha(L).
    \end{align*}
    Applying $\pi$ and using its properties 1 and 2 we obtain
    \begin{align*}
        \pi(u) \geq \pi(u_L) \geq \pi (u) + \alpha(L).
    \end{align*}
For $x,y \in \Omega$, since $\pi(u_k)$ is also $L$-Lipschitz, we can write 
\begin{align*}
    \pi(u)(x) - \pi(u)(y) \leq \pi(u_L)(x) - \pi(u_L)(y)- \alpha(L) \leq L|x-y| - \alpha(L).
\end{align*}
Since $\omega$ is concave, we define $\partial^+ \omega$ its super-differential. Let us choose $L$ such that $L\in \partial^+ \omega(|x-y|)$ (note that $\omega$ is also non-decreasing, so we can take $L\ge 0$; also note that the superdifferential is non-empty when $r>0$ because $\omega$ is finite on $\R_+$). This implies that $r\mapsto Lr-\omega(r)$ (which is a convex function) is minimized at $r=|x-y|$ because $0\in \partial(r\mapsto Lr-\omega(r))$. Then we obtain for this $L$ the equality $\alpha(L)=L|x-y| -\omega(|x-y|)$, and thus
\begin{align*}
      \pi(u)(x) - \pi(u)(y) \leq \omega(|x-y|),
\end{align*}
so that $\pi(u)$ admits $\omega$ as a modulus of continuity.
\end{proof}

Using the above theorem, we can then prove the following:

\begin{theorem}
\label{wreg}
    Let $\p_0 \in L^1(\Omega)$ be such that $\log(\p_0)$ admits $\omega$ as a modulus of continuity, then for all $k$ and $h>0$, $\p_k ^h$ also admits $\omega$ as a modulus of continuity. 
\end{theorem}

\begin{proof}
    One only has to apply \Cref{proofOperator} with $\pi = \log \circ \, \Pi \circ \exp$ to the $L^\infty$ function $\log(\p_0)$, where $\Pi$ is defined in \cref{propsPi}. Assumptions 1 and 2 are satisfied because of \Cref{propsPi} since the $\log$ function is increasing and satisfies $\log(ab)=\log(a) + \log(b)$ for $a,b >0$. The last assumption is satisfied because of the first item of \Cref{UnifEstimates}. We therefore obtain that $\p_1 ^h$ admits $\omega$ as a modulus of continuity, and iterating the above at each step of the JKO scheme gives the result.
\end{proof}

We remind the following theorem from \cite{Agueh} which proves the convergence of the JKO scheme, under the additional conditions on the cost function $c$ :

\begin{enumerate}
    \item $\lim_{|z| \to +\infty} \frac{c(z)}{|z|}=+ \infty$,
    \item there exists $\alpha,\beta >0$ such that $\alpha |z|^p \leq c(z) \leq \beta (1+ |z|^p)$.
\end{enumerate}

\begin{theorem}
\label{thmAgueh}
    Suppose $\log(\p_0) \in L^{\infty} (\Omega)$, and denote by $\p^h$ the piecewise constant (in time) interpolation of $(\p_k ^h)_k$. Then there exists $\p\in L^1 (\Omega)$ such that $\log(\p)\in L^\infty$ and such that $\p_h$ strongly converges to $\p$ in $L^1 (\Omega)$, and $\p$ is the unique weak solution of
    \begin{align}
    \label{AguehPDE}
        \partial_t \p = \nabla \cdot (\p \nabla c^*(\nabla \log(\p)))
    \end{align}
\end{theorem}

Another proof of convergence, in the case where the cost $c$ is exactly a power, is provided in \cite{CailletSantambrogioDoubly}. In this reference the bound from below and above on the initial density $\p_0$ is removed, but only the existence is proven, not the uniqueness.

From \Cref{wreg} and \Cref{thmAgueh} we deduce the following :
\begin{theorem}
    Let $\p_0 \in L^1(\Omega)$ be non negative. If $\log(\p_0)$ admits $\omega$ as a modulus of continuity, then for all $t>0$, $\p(t)$ admits $\omega$ as a modulus of continuity, where $\p$ is the solution of \cref{AguehPDE}.
\end{theorem}
\begin{proof}
    Since $\log(\p_0)$ is continuous, we can construct the unique solution of \cref{AguehPDE} using \Cref{thmAgueh}. Applying \Cref{wreg}, we have for all $k$ that $\log(\p^h _k)$ admits $\omega$ as a modulus of continuity, and therefore for all $t\geq 0$, we have
    \begin{align*}
        |\log(\p^h (t,x)) - \log(\p^h (t,y))| \leq \omega(|x-y|).
    \end{align*}
    Passing to the pointwise almost everywhere limit $\p^h (t) \to \p(t)$ as $h\to 0$ in the above equation allows us to conclude that $\p(t)$ admits $\omega$ as a modulus of continuity
\end{proof}

When we have uniqueness to \cref{AguehPDE}, which is the case for example is we look at bounded solutions, we can show that the generalized Fisher information decreases along the flow :
\begin{theorem}
    Let $\p_0 \in L^1(\Omega)$ be non negative and such that $\log(\p_0) \in L^\infty (\Omega)$. Denote by $\p$ the solution of \cref{AguehPDE}, then the generalized Fisher information
    \begin{align*}
        \int H (\nabla \log \p) \diff \p 
    \end{align*}
    is non increasing in time.
\end{theorem}
\begin{proof}
    From \Cref{UnifEstimates,thmAgueh}, we know that the discrete solution $(\tilde{\p}^h _k)_k$ constructed from the JKO scheme started from $\p(t)$ for any $t\geq 0$ satisfies 
    \begin{align*}
         \int H (\nabla \log \p(t)) \diff \p(t)  \geq   \int H (\nabla \log \tilde{\p}^h _k) \diff \tilde{\p}^h _k.
    \end{align*}
    By uniqueness of the solution, and the strong convergence of the discrete solutions provided by \cref{thmAgueh}, we have $\Tilde{\p}^h (s) \to \p(t+s)$ and the lower semi continuity of the right hand side (because $H$ is convex) gives
    \begin{align*}
        \int H (\nabla \log \p(t)) \diff \p(t) \geq \int H (\nabla \log \p(t +s)) \diff \p(t +s)
    \end{align*}
    for any $s\geq 0$.
\end{proof}

\begin{remark}
    Doing this computation directly at the continuous level by differentiating the generalized Fisher Information and trying to show that its derivative is non-positive, seems to be, in our opinion, rather complicated. In particular, after some simplification we obtain the following computations (subject to caution and ignoring boundary terms), where we denote $u=\log\p$. If $\p$ solves \cref{AguehPDE} then the equation solved by $u$ is
    $$\partial_t u = K_{c^*}+D_{c^*},\mbox{ where }D_{c^*}:=\nabla\cdot(\nabla c^*(\nabla u))=\sum_{j,k}c^*_{jk}(\nabla u)u_{jk},\; K_{c^*}:=\nabla u\cdot\nabla c^*(\nabla u)=\sum_i u_i c^*_i(\nabla u).$$
    We also use $\partial_t\p=\p\partial_t u$ and with an integration by parts we obtain
    \begin{eqnarray*}
    \frac{\diff}{\diff t}\int \p H(\nabla u)&=&\int\partial_t\p H(\nabla u)+\p H_i(\nabla u) \partial_t u_i=\int (\p H(\nabla u)-(\p H_i(\nabla u))_i)\partial_t u\\
    &=&\int \p(K_{c^*}+D_{c^*})(H(\nabla u)-K_H-D_H),
    \end{eqnarray*}
    where $K_H$ and $D_H$ are defined as for $K_{c^*}$ and $D_{c^*}$, i.e. $D_{H}:=\sum_{j,k}H_{jk}(\nabla u)u_{jk},\; K_{H}:=\sum_i u_i H_i(\nabla u) $. We do not find it easy to see that the r.h.s. is negative (even in the case $H=c^*$). On the other hand, the Lipschitz regularity result contained in Theorem \ref{UnifEstimates} can indeed be recovered by a maximum principle using the equation satisfied by $u=\log\p$: differentiating it provides 
    $$\partial_t u_h=c^*_{jkl}(\nabla u)u_{lh}u_{jk}+ c^*_{jk}(\nabla u)u_{jkh}+u_{ih} H_i(\nabla u)+u_i H_{il}(\nabla u)u_{lh}.$$
    Multipliying by $u_h$ and looking at $\max |\nabla u|^2$ allows to obtain
    $$\frac{\diff}{\diff t}\max|\nabla u|^2\leq 0.$$
    Yet, of course, this formal compuation would require to be suitably justified.
\end{remark}

\subsection{Extension to relativistic costs}
In \cite{McCPue}, the authors generalize the construction from \cite{Agueh} to "relativistic" cost functions, meaning cost functions of the form
\begin{align*}
    c(z)=\begin{cases}
        \Tilde{c}(z) \quad &\text{if } |z|\leq 1 \\
        + \infty &\text{if } |z| >1,
    \end{cases}
\end{align*}
where $\Tilde{c} \in C([0,1]) \cap C^2 ([0,1))$ is a strictly convex function, with $|\nabla \Tilde{c}(z)| \to + \infty$ when $|z| \to 1$. The main equation of interest entering in this framework is the so called Relativistic Heat Equation 
\begin{align*}
    \partial_t \p = \nabla \cdot \left( \p \,\frac{\nabla \p}{\sqrt{\p ^2 + |\nabla \p|^2}} \right),
\end{align*}
where $\Tilde{c}(z) = 1 - \sqrt{1 - |z|^2}$. Their construction relies on an approximation argument of the irregular cost function $c$, by smoother cost functions $c_\varepsilon$ by inf-convolution, which allows us to apply our previous results before passing to the limit. 
\begin{lemma}
    Define 
    \begin{align*}
        c_\varepsilon (z) = \inf_{ w \in \R ^d } \left( c(z-w) + \frac{|w|^2}{2\varepsilon} \right),
    \end{align*}
    it is a strictly convex and radially symmetric $C^2(\R^d)$ function. If we denote by $\p^h _\varepsilon$ the solution of \cref{jkoproblem} with the cost $c_\varepsilon$, then as $\varepsilon \to 0$, $\p^h _\varepsilon$ weakly converges to the solution $\p^h$ of \cref{jkoproblem} with the cost $c$.
\end{lemma}
\begin{proof}
    See Section 3.2 and Lemma 3.2 from \cite{McCPue}.
\end{proof}
As in the prevous section, we recursively define the sequence $(\p^h _k)_k$. And we have the same results as before :

\begin{theorem}
\label{estimatesrelativistic}
    Let $\p_0\in L^1(\Omega)$. The following estimates are uniform in $h$ and $k$:
    \begin{enumerate}
        \item if $\log(\p_0)$ is $L$-Lipschitz, then $\log(\p_k ^h)$ is $L$-Lipschitz;
        \item for $p>1$, if $\p_0^{1/p} \in W^{1,p} (\Omega)$, then $(\p_k ^h)^{1/p} \in W^{1,p} (\Omega)$ and $\Vert (\p_k ^h)^{1/p} \Vert_{W^{1,p}} \leq \Vert \p_0^{1/p} \Vert_{W^{1,p}}$;
        \item if  $\p_0 \in BV(\Omega)$, then $\p_k ^h \in BV(\Omega)$ and $\Vert \p_k ^h \Vert_{BV} \leq \Vert \p_0 \Vert_{BV}$.
        \item if $\p_0 \in W^{1,1} (\Omega)$, all $\p_k ^h$ belong to a weakly compact subset of $W^{1,1} (\Omega)$.
        \item if $\log(\p_0)$ admits $\omega$ as a modulus of continuity, then all $\p_k ^h$ also admit $\omega$ as a modulus of continuity.
    \end{enumerate}
\end{theorem}
\begin{proof}
    We apply \Cref{UnifEstimates,wreg} from the previous section to the approximated problem with $c_\varepsilon$, and then pass to the limit $\varepsilon \to 0$. All estimates pass to the limit by lower semi continuity with respect to the weak convergence.
\end{proof}
Using the convergence result from \cite{McCPue}, we have the following :
\begin{theorem}
    Let $\p_0 \in L^1(\Omega)$ be such that $\log(\p_0)$ admits $\omega$ as a modulus of continuity. The interpolated sequence $(\p^h)$ strongly converges to $\p$, a solution in the sense of distributions of 
    \begin{align}
        \partial_t \p = \nabla \cdot (\p \nabla c^*(\nabla \log(\p))),
    \end{align}
which is such that $\log(\p(t))$ admits $\omega$ as a modulus of continuity for all $t$.
\end{theorem}
\begin{proof}
    We use Theorem 1.8 from \cite{McCPue} to have the convergence of the discrete sequence, and use the last point \Cref{estimatesrelativistic} to prove that the modulus of continuity of $\log(\p_0)$ is propagated. Indeed, for the constant in time interpolation $\p^h$, we have for all $t \geq 0$
    \begin{align*}
        |\p^h (t,x) - \p(t,y) | \leq \omega (|x-y|), 
    \end{align*}
    and we can pass to the limit $h \to 0$.
\end{proof}
\bibliographystyle{plain}
\bibliography{biblio}
\end{document}